\documentclass[pdftex,a4paper,abstracton]{scrartcl}
\usepackage[utf8]{inputenc}
\usepackage[english]{babel}
\usepackage{amssymb,amsmath,amsthm}
\usepackage[round]{natbib}
\usepackage{hyperref}
\usepackage[flushmargin]{footmisc}

\newtheorem{Satz}{Theorem}
\newtheorem*{Satz*}{Theorem A}
\newtheorem{Lemma}{Lemma}
\theoremstyle{definition}
\newcommand{\RR}{\mathbb R}
\newcommand{\NN}{\mathbb N}

\allowdisplaybreaks
\begin{document}
\nocite{*}

\title{Weak Convergence of the Weighted Sequential Empirical Process of some Long-Range Dependent Data}
\author{Jannis Buchsteiner\thanks{~E-mail: \texttt{jannis.buchsteiner@rub.de}\newline Research supported by Collaborative Research Center SFB 823 {\em
Statistical modeling of nonlinear dynamic processes.}\newline NOTICE: this is the author's version of a work that was accepted for publication in Statistics and Probability Letters. Changes resulting from the publishing process, such as peer review, editing, corrections, structural formatting, and other quality control mechanisms may not be reflected in this document. Changes may have been made to this work since it was submitted for publication. A definitive version was subsequently published in \textit{Statistics and Probability Letters} (2015), 96, pp. 170-179.} \\ \normalsize{\textit{Fakultät für Mathematik, Ruhr-Universität Bochum,
Germany.}}}
\date{}
\maketitle

\begin{abstract}
 Let $(X_k)_{k\geq1}$ be a Gaussian long-range dependent process with $EX_1=0$, $EX_1^2=1$ and covariance function $r(k)=k^{-D}L(k)$.
For any measurable function $G$ let $(Y_k)_{k\geq1}=(G(X_k))_{k\geq1}$. We study the asymptotic behaviour of the associated sequential empirical process
$\left(R_N(x,t)\right)$ with respect to a weighted sup-norm $\|\cdot\|_w$. We show that, after an appropriate normalization, $\left(R_N(x,t)\right)$ converges
weakly in the space of càdlàg functions with finite weighted norm to a Hermite process.\\

\end{abstract}

\section{Introduction}
Given a stationary stochastic process $(Y_j)_{j\geq 1}$, with marginal distribution function $F(x)=P(Y_1\leq x)$, we define the sequential empirical process
\begin{equation*}
 R_N(x,t)=\sum_{j=1}^{\lfloor Nt\rfloor} \left( 1_{\{Y_j\leq x  \}} -F(x)  \right),\, x\in \RR, 0\leq t\leq 1.
\end{equation*}
This process plays an important role in statistics, e.g. in the study of nonparametric change-point tests. The asymptotic distribution of
the sequential empirical process was initially determined by \citet{Mueller}, and independently \citet{Kiefer}, who both studied the case when the underlying
data $(Y_j)_{j\geq 1}$ are independent and identically distributed. In this case,  $N^{-1/2} R_N(x,t)$ converges in distribution towards a mean-zero Gaussian
process $K(x,t)$ with covariance structure $E(K(x,s)K(y,t))=(s\wedge t) (F(x\wedge y)-F(x)\, F(y))$. The process $K(x,t)$ is also called a Kiefer-Müller
process. \citet*{Major} proved an almost sure approximation theorem for the sequential empirical process with sharp rates, again in the case of i.i.d. data.

Sequential empirical processes of dependent data have been studied by a large number of authors, e.g.
\citet{Berkes77} and \citet{Philipp} for strongly mixing processes, and \citet*{Berkes} for so called S-mixing processes. For long-range dependent data, the
sequential empirical process was first studied by \citet{Dehling}, in the case of a Gaussian subordinated process. \citet{Giraitis} used similar techniques to
establish weak convergence if the underlying data is a long memory moving average process.\\
Under some technical conditions, \citet{Dehling} prove convergence of the normalized sequential empirical process in the space
$D([-\infty,\infty]\times [0,1])$  towards a process of the type $ J(x) Z(t)$, $x\in \RR, 0\leq
t\leq 1$, where $J:\RR\rightarrow \RR$ is a deterministic function and $(Z(t))_{0\leq t\leq 1}$ is a Hermite process.

In the present paper, we consider the above result with regard to the weighted sequential empirical process $w(x)R_N(x,t)$, where $w(x)=(1+|x|)^\lambda$, for
some $\lambda>0$. Therefore we equip the function space 
\begin{equation*}
  D_w([-\infty,\infty]\times [0,1]):= \{f\in D([-\infty,\infty]\times [0,1]): \sup_{x\in\RR,t\in[0,1]} |w(x)f(x,t)| <\infty \},
\end{equation*}
with the weighted sup-norm $\|f\|_w:=\sup|w(x)f(x,t)|$ and show that the result of Dehling and Taqqu takes place in this normed subspace of
$D([-\infty,\infty]\times [0,1])$.

The asymptotic distribution of the weighted one-parameter empirical process $(R_N(x,1))$ has been studied for i.i.d. data by \citet{Cibisov} and
\citet{O'Reilly}. \citet{Shao} treated the cases when the underlying data are strong mixing, $\rho$-mixing and associated. Recently, \citet*{Wu} studied
empirical process convergence with
respect to weighted norms for linear long-range dependent data. 

Weak convergence of the empirical process with respect to weighted supremum norms has been applied by \citet{Beutner} in their study of the
asymptotic behaviour of the distortion risk measure. They developed a modified functional delta method (MFDM) which requires only quasi-Hadamard
differentiability on the one hand, but weighted convergence of the empirical process on the other hand. By using the MFDM, \citet{Beutner12} also determined the
asymptotic distribution of U- and V-statistics with an unbounded kernel. The weight functions arising in this context are functions of $x$ only. More generally
one could study weight functions $w(x,t)$. However, this is beyond the scope of the present paper.

\section{Definitions and Main Results}
We consider a stationary Gaussian process $(X_j)_{j\geq 1}$ with $EX_1=0$, $EX_1^2=1$ and covariance function $r(k)=EX_1X_{k+1}$, which satisfies
\begin{equation}\label{covariance}
r(k)=k^{-D}L(k),
\end{equation}
where $L$ is a slowly varying function at infinity and $0<D<1$. Such a sequence is called a Gaussian long-range dependent process. For any measurable function
$G:\RR\rightarrow\RR$ we define the subordinated process $(Y_j)_{j\geq1}$ by
\begin{equation*}
 Y_j:=G(X_j).
\end{equation*}
A useful tool to establish weak convergence of $(R_N(x,t))$ under these circumstances is the collection of Hermite polynomials. The Hermite polynomial $H_n$ of order $n$ is
defined as
\begin{equation*}
 H_n(x):=(-1)^ne^{x^2/2}\frac{d^n}{dx^n}e^{-x^2/2}.
\end{equation*}
For example $H_0(x)=1$, $H_1(x)=x$ and $H_2(x)=x^2-1$. Since $(H_n)_{n\geq0}$ is an orthogonal basis for the space of square integrable functions with respect
to the standard normal distribution, we have for any $x\in\RR$ the series expansion
\begin{equation}\label{HD}
 1_{\{Y_j\leq x\}}-F(x)=\sum_{q=0}^{\infty}\frac{J_q(x)}{q!}H_q(X_j).
\end{equation}
As usual, the Hermite coefficients $J_q(x)$ are given by the inner product, i.e.
\begin{equation*}
 J_q(x)=E(1_{\{Y_j\leq x\}}-F(x))H_q(X_j)=E1_{\{Y_j\leq x\}}H_q(X_j)=\int\limits_{\{G(s)\leq x\}}H_q(s)\varphi(s)ds,
\end{equation*}
for $q\geq1$, where $\varphi$ is the standard normal density. With regard to \eqref{HD} we call the index $m(x)$ of the first nonzero Hermite coefficient the
Hermite rank of ${1_{\{G(\cdot)\leq x\}}-F(x)}$. Since $E(1_{\{Y_j\leq x\}}-F(x))=0$ we have $m(x)\geq1$. If $0<D<1/m(x)$, then ${(1_{\{Y_j\leq
x\}}-F(x))_{j\geq1}}$ exhibits long-range dependence, see \citet{Taqqu1975}.\\
Moreover we set $m:=\min\{m(x):x\in\RR\}$ and call $m$ the Hermite rank of the class of functions $\{1_{\{G(\cdot)\leq x\}}-F(x):x\in\RR\}$. 
\begin{Satz*}[{\citealt[Theorem 1.1]{Dehling}}]
 Let $(X_j)_{j\geq 1}$ be a stationary, mean-zero Gaussian process with covariance \eqref{covariance}, let the class of functions $1_{\{G(X_j)\leq x\}}-F(x),
-\infty<x<\infty$, have Hermite rank m and let $0<D<1/m$. Then
\begin{equation*}
 \left\{d_N^{-1}R_N(x,t): -\infty\leq x\leq\infty, 0\leq t\leq 1\right\}
\end{equation*}
converges weakly in $D([-\infty,\infty]\times[0,1])$, equipped with the sup-norm, to
\begin{equation*}
 \left\{\frac{J_m(x)}{m!}Z_m(t): -\infty\leq x\leq\infty, 0\leq t\leq 1\right\}.
\end{equation*}
\end{Satz*}
The normalization factor $d_N$ is asymptotically proportional to $\sqrt{N^{2-mD}L^m(N)}$, more precisely
\begin{equation*}
 d_N^2=\operatorname{Var}\left(\sum_{j=1}^N H_m(X_j)\right),
\end{equation*}
see \citet[Corollary 4.1]{Taqqu1975}. The process $(Z_m(t))_{t\in[0,1]}$ is called an $m$th order Hermite process. It can be represented as a
multiple Wiener-It\^o integral as well as a Wiener-It\^o-Dobrushin integral, see \citet{Taqqu1979}. For $m=1$ it is a fractional Brownian motion and therefore
Gaussian, but it is non Gaussian for $m\geq2$.

Heuristically, we have to control $w(x)F(x)$ and $w(x)(1-F(x))$ for $x\rightarrow-\infty$ resp. $x\rightarrow\infty$ to get a weighted version of Theorem A.
Therefore we require that $F$ has at least a finite $\delta$-th moment, i.e. 
\begin{equation}\label{delta}
 \int|x|^{\delta}dF(x)<\infty
\end{equation}
for some $\delta>0$.
\begin{Satz}\label{nclt}
 Let $(X_j)_{j\geq 1}$ be a stationary, mean-zero Gaussian process with covariance \eqref{covariance}, let the class of functions $1_{\{G(X_j)\leq x\}}-F(x),
-\infty<x<\infty$, have Hermite rank m and let $0<D<1/m$. If $F$ has a finite $\delta$-th moment then
\begin{equation*}
 \left\{d_N^{-1}R_N(x,t): -\infty\leq x\leq\infty, 0\leq t\leq 1\right\}
\end{equation*}
converges weakly in $D_w([-\infty,\infty]\times [0,1])$, equipped with the weighted sup-norm $\|\cdot\|_w$, to
\begin{equation*}
 \left\{\frac{J_m(x)}{m!}Z_m(t): -\infty\leq x\leq\infty, 0\leq t\leq 1\right\},
\end{equation*}
where $w(x)=(1+|x|)^{\lambda}$ and $\lambda=\delta/3$.
\end{Satz}
If we want to use Theorem \ref{nclt} to apply the MFDM, we need $\lambda>1$, i.e. the distribution function $F$ must have a finite $\delta$-th moment with $\delta>3$. We
conjecture that the choice $\lambda=\delta/3$ could be improved to $\delta/2$, since $\lambda=\delta/3$ is only necessary to get \eqref{beschraenkt1} and the rest of the proof works for $\lambda=\delta/2$.\\
To prove Theorem \ref{nclt} we  need a weighted version of Taqqu's weak reduction principle \citep[cf.][]{Taqqu1975, Dehling}. 
\begin{Satz}\label{weak-reduction}
Under the assumptions of Theorem \ref{nclt} there exist constants $C, \kappa>0$ such that for any
$0<\varepsilon\leq1$
\begin{align}
 P\Biggl(&\max_{n\leq N}\sup_{-\infty\leq x\leq\infty}d_N^{-1}\left|w(x)\sum_{j=1}^n\left(1_{\{Y_j\leq x\}}-F(x)-\frac{J_m(x)}{m!}H_m(X_j)\right)\right|
         >\varepsilon\Biggr)\notag\\
\leq &CN^{-\kappa}(1+\varepsilon^{-3}),\label{reduction}
\end{align}
where $w(x)=(1+|x|)^{\lambda}$ and $\lambda=\delta/3$.
\end{Satz}

\section{Proofs}
From now on we assume that the conditions of Theorem \ref{nclt} are satisfied. Especially let $w(x)=(1+|x|)^{\lambda}$ with $\lambda=\delta/3$. For consistency
reasons we adopt some notations by Dehling and Taqqu, namely
\begin{align*}
\Lambda(x)&:=F(x)+\int 1_{\{G(s)\leq x\}}\frac{|H_m(s)|}{m!}\varphi(s)ds, \\
 S_N(n,x)&:=d_N^{-1}\sum_{j=1}^n\left(1_{\{Y_j\leq x\}}-F(x)-\frac{J_m(x)}{m!}H_m(X_j)\right).
\end{align*}
Furthermore for $x\leq y$ we set 
\begin{align*}
F(x,y):&=F(y)-F(x),\\
J_m(x,y):&=J_m(y)-J_m(x)\\
S_N(n,x,y):&=S_N(n,y)-S_N(n,x)\\
\Lambda(x,y):&=\Lambda(y)-\Lambda(x).
\end{align*} 
Note that $\Lambda$ is nondecreasing and that $\Lambda(x,y)$ bounds $F(x,y)$ as well as $(1/m!)J_m(x,y)$ if $x\leq y$.

Lemma \ref{momente} is a modification of Lemma 3.1 by Dehling and Taqqu. The following rearrangement is small but necessary. 
\begin{Lemma}\label{momente}
 Under the assumptions of Theorem \ref{nclt} there exist constants $\gamma>0$ and $C$ such that for $n\leq N$,
\begin{equation}\label{momenteungl}
 E\left|S_N(n,x,y)\right|^2\leq C\left(\frac{n}{N}\right)N^{-\gamma}F(x,y)\left(1-F(x,y)\right).
\end{equation}
\end{Lemma}
We can bound \eqref{momenteungl} again by $C(n/N)N^{-\gamma}(1-F(y))$, or $C(n/N)N^{-\gamma}F(x)$, which is useful for $y\rightarrow\infty$ resp.
$x\rightarrow-\infty$. During this paper we will handle $C$ as a universal constant, possibly growing from line to line and from lemma to lemma, but at the end
bounded and independent of $N,n,x$ and $\varepsilon$.
\begin{proof}
The Hermite expansion 
\begin{equation*}
\sum_{q=m}^{\infty}\frac{J_q(x,y)}{q!}H_q(X_j)=1_{\{x\leq Y_j\leq y\}}-F(x,y) 
\end{equation*}
yields
\begin{equation*}
 \sum_{q=m}^{\infty}\frac{J_q^2(x,y)}{q!}=E\left(1_{\{x\leq Y_j\leq y\}}-F(x,y)\right)^2=F(x,y)\left(1-F(x,y)\right).
\end{equation*}
Together with $EH_q(X_j)H_q(X_k)=q!(EX_jX_k)^q=q!(r(j-k))^q$ we get
\begin{align*}
 &E\left(\sum_{j\leq n}\left(1_{\{x\leq Y_j\leq y\}}-F(x,y)-\frac{J_m(x,y)}{m!}H_m(X_j)\right)\right)^2\\
=&E\left(\sum_{j\leq n}\sum_{q=m+1}^{\infty}\frac{J_q(x,y)}{q!}H_q(X_j)\right)^2\\
=&\sum_{q=m+1}^{\infty}\frac{J_q^2(x,y)}{q!}\frac{1}{q!}\sum_{j,k\leq n}EH_q(X_j)H_q(X_k)\\
\leq&F(x,y)(1-F(x,y))\sum_{j,k\leq n}\left|r(j-k)\right|^{m+1}.
\end{align*}
Since $\sum_{j,k\leq n}|r(j-k)|^{m+1}\leq 2n\sum_{k=1}^n k^{-D(m+1)}|L(k)|^{m+1}$, we have
\begin{align*}
 \sum_{j,k\leq n}|r(j-k)|^{m+1}&\leq Cn^{2-D(m+1)}|L(n)|^{m+1},~~~\textnormal{ for $D(m+1)<1$,}\\
 \sum_{j,k\leq n}|r(j-k)|^{m+1}&\leq Cn,~~~\textnormal{ for $D(m+1)>1$,}\\
 \sum_{j,k\leq n}|r(j-k)|^{m+1}&\leq Cn^{1+\alpha}|L(n)|^{m},~~~\textnormal{ for $D(m+1)=1$}
\end{align*}
and $0<\alpha<1-mD$. In general we get 
\begin{equation*}
 \sum_{j,k\leq n}|r(j-k)|^{m+1}\leq Cn^{1+\alpha\vee2-D(m+1)}L'(n),
\end{equation*}
where $L'$ is some suitable slowly varying function. Therefore
\begin{align*}
 E|S_N(n,x,y)|^2&\leq Cd_N^{-2}F(x,y)(1-F(x,y))n^{1+\alpha\vee2-D(m+1)}L'(n)\\
&\leq CF(x,y)(1-F(x,y))n^{1+\alpha\vee2-D(m+1)}N^{mD-2}L'(n)\left(L(N)\right)^{-m}\\
&= CF(x,y)(1-F(x,y))\left(\frac{n}{N}\right)^{1+\alpha\vee2-D(m+1)}N^{mD+\alpha-1\vee -D}L'(n)\left(L(N)\right)^{-m}\\
&\leq CF(x,y)(1-F(x,y))\left(\frac{n}{N}\right)N^{-\gamma}.
\end{align*}
\end{proof}
\begin{Lemma}\label{lemma3}
Under the assumptions of Theorem \ref{nclt} there exist constants $\rho>0$ and $C$ such that for
any $n\leq N$ and $0<\varepsilon\leq 1$,
\begin{equation*}
 P\left(\sup_{x\in\RR}|w(x)S_N(n,x)|>\varepsilon\right)\leq CN^{-\rho}\left(\frac{n}{N}\varepsilon^{-3}+\left(\frac{n}{N}\right)^{2-mD}\right),
\end{equation*}
where $w(x)=(1+|x|)^{\lambda}$ and $\lambda=\delta/3$.
\end{Lemma}
\begin{proof}
As \citet[Lemma 3.2]{Dehling} we will use the classical chaining technique. For simplicity we will bound the probability separately for $x\in[0,\infty)$ and
$x\in(-\infty,0]$, starting with the first case. Since
$\lim_{x\rightarrow\infty}w(x)\Lambda(x)=\infty$, the refining partitions $(x_i(k))_{i\in\NN}$ of $[0,\infty)$ should consist of an infinite number of grid
points. For $k\geq 0$ we set
\begin{equation*}
x_i(k):=\inf\{x\geq0:w(x)\Lambda(x)\geq \Lambda(0)+i2^{-k}\}. 
\end{equation*}
By this definition we have
\begin{align}
 &w(x_{i+1}(k))\Lambda(x_i(k),x_{i+1}(k)-)\notag\\
\leq&w(x_{i+1}(k))\Lambda(x_{i+1}(k)-)-w(x_{i}(k))\Lambda(x_{i}(k))\notag\\
\leq&2^{-k}.\label{abstand}
\end{align}
Moreover, using condition \eqref{delta} together with the assumption $\delta=3\lambda$ and $i+1\leq \Lambda(\infty)w(x_{i+1}(0))$ we get
\begin{align}
 &\sum_{j=0}^{\infty}w(x_{j+1}(0))^2(1-F(x_j(0)))\notag\\
=&\sum_{j=0}^{\infty}\sum_{i=j}^{\infty}w(x_{j+1}(0))^2(F(x_{i+1}(0))-F(x_i(0)))\notag\\
=&\sum_{i=0}^{\infty}\sum_{j=0}^{i}w(x_{j+1}(0))^2(F(x_{i+1}(0))-F(x_i(0)))\notag\\
\leq&\sum_{i=0}^{\infty}(i+1)w(x_{i+1}(0))^2(F(x_{i+1}(0))-F(x_i(0)))\notag\\
\leq&\Lambda(\infty)\sum_{i=0}^{\infty}w(x_{i+1}(0))^3(F(x_{i+1}(0))-F(x_i(0)))\notag\\
\leq&C\sum_{i=0}^{\infty}w(x_{i}(0))^3(F(x_{i+1}(0))-F(x_i(0)))\notag\\
<&\infty.\label{beschraenkt1}
\end{align}
Notice that for all $k\in\NN$ $(x_j(k+1))_{j\in\NN}$ is a refinement of $(x_i(k))_{i\in\NN}$ and so for any index $i\in\NN$ it exists an index $j\in\NN$ with
$x_j(k+1)=x_i(k)$ and $x_{j-2}(k+1)=x_{i-1}(k)$. This yields 
\begin{align}
&w(x_i(k))^2(F(x_i(k))-F(x_{i-1}(k)))\notag\\
=&w(x_j(k+1))^2(F(x_j(k+1))-F(x_{j-2}(k+1)))\notag\\
=&w(x_j(k+1))^2(F(x_j(k+1))-F(x_{j-1}(k+1)))\notag\\
{}&+w(x_j(k+1))^2(F(x_{j-1}(k+1))-F(x_{j-2}(k+1)))\notag\\
\geq&w(x_j(k+1))^2(F(x_j(k+1))-F(x_{j-1}(k+1)))\notag\\
{}&+w(x_{j-1}(k+1))^2(F(x_{j-1}(k+1))-F(x_{j-2}(k+1))).\label{abschaetzung}
\end{align}
Since \eqref{abschaetzung} implies
\begin{align*}
&\sum_{i=1}^{\infty}w(x_{i}(k+1))^2(F(x_{i}(k+1))-F(x_{i-1}(k+1)))\\
\leq&\sum_{i=1}^{\infty}w(x_{i}(k))^2(F(x_{i}(k))-F(x_{i-1}(k)))
\end{align*}
and \eqref{abstand} implies
\begin{align*}
 w(x_{i+1}(k))&\leq\frac{1}{\Lambda(0)}\Lambda(x_{i+1}(k)-)w(x_{i+1}(k))\\
&\leq\frac{1}{\Lambda(0)}\left(2^{-k}+w(x_i(k))\Lambda(x_i(k))\right)\\
&\leq\frac{1}{\Lambda(0)}\left(1+w(x_i(k))\Lambda(\infty)\right)\\
&\leq Cw(x_i(k))
\end{align*}
we get
\begin{align}
    &\sum_{i=1}^{\infty}w(x_{i+1}(k+1))^2(F(x_{i+1}(k+1))-F(x_{i-1}(k+1)))\notag\\
=&\sum_{i=1}^{\infty}w(x_{i+1}(k+1))^2(F(x_{i+1}(k+1))-F(x_{i}(k+1)))\notag\\
{}&+\sum_{i=1}^{\infty}w(x_{i+1}(k+1))^2(F(x_{i}(k+1))-F(x_{i-1}(k+1)))\notag\\
\leq& C\sum_{i=1}^{\infty}w(x_{i}(k+1))^2(F(x_{i}(k+1))-F(x_{i-1}(k+1)))\notag\\
\leq& C\sum_{i=1}^{\infty}w(x_{i}(k))^2(F(x_{i}(k))-F(x_{i-1}(k)))\notag\\
\leq& C\sum_{i=1}^{\infty}w(x_{i}(0))^2(F(x_{i}(0))-F(x_{i-1}(0)))\notag\\
<&\infty,\label{beschraenkt2}
\end{align}
where \eqref{beschraenkt2} is uniform in $k$. We will use \eqref{abstand}, \eqref{beschraenkt1} and \eqref{beschraenkt2} as follows. For any $x\geq0$ and any
$k\in\{1,\ldots.K\}$ there exists an index $i_k(x)$
such that
\begin{equation*}
 x_{i_k(x)}(k)\leq x<x_{i_k(x)+1}(k).
\end{equation*}
This nesting yields a stepwise chaining of $x$, given by
\begin{equation*}
 0\leq x_{i_0(x)}(0)\leq x_{i_1(x)}(1)\leq\ldots\leq x_{i_K(x)}(K)\leq x.
\end{equation*}
Using the grid points above, we get
\begin{align}
|w(x)S_N(n,x)|\leq&|w(x)S_N(n,x_{i_0(x)}(0))|+|w(x)S_N(n,x_{i_0(x)}(0),x_{i_1(x)}(1))|\notag\\ 
                  &+\ldots+|w(x)S_N(n,x_{i_K(x)}(K),x)|\notag\\
              \leq&|w(x_{i_0(x)+1}(0))S_N(n,x_{i_0(x)}(0))|+|w(x_{i_1(x)+1}(1))S_N(n,x_{i_0(x)}(0),x_{i_1(x)}(1))|\notag\\
                  &+\ldots+|w(x)S_N(n,x_{i_K(x)}(K),x)|.\label{abschaetzung2}
\end{align}
The last term of the right hand side can be bounded as follows
\begin{align}
\left|w(x)S_N(n, x_{i_K(x)}(K),x)\right|&=d_N^{-1}\Biggl|\sum_{j\leq n}\biggl(w(x)\Bigl(1_{\{x_{i_K(x)}(K)< Y_j\leq x\}}-F(x_{i_K(x)}(K),x)\Bigr)\notag\\ 
   &{}\hspace{1,65cm} -w(x)\frac{J_m(x_{i_K(x)}(K),x)}{m!}H_m(X_j)\biggr)\Biggr|\notag\\
&\leq d_N^{-1}\sum_{j\leq n}\Bigl(w(x_{i_K(x)+1}(K))1_{\{x_{i_K(x)}(K)< Y_j< x_{i_K(x)+1}(K)\}}\notag\\
&{}\hspace{1,5cm}  +w(x_{i_K(x)+1}(K))F(x_{i_K(x)}(K),x_{i_K(x)+1}(K)-)\Bigr)\notag\\
&+w(x_{i_K(x)+1}(K))\Lambda(x_{i_K(x)}(K),x_{i_K(x)+1}(K)-)d_N^{-1}\left|\sum_{j\leq n}H_m(X_j)\right|\notag\\
&\leq\left|w(x_{i_K(x)+1}(K))S_N(n,x_{i_K(x)}(K),x_{i_K(x)+1}(K)-)\right|\notag\\
&+2nd_N^{-1}w(x_{i_K(x)+1}(K))F(x_{i_K(x)}(K),x_{i_K(x)+1}(K)-)\notag\\
&+2w(x_{i_K(x)+1}(K))\Lambda(x_{i_K(x)}(K),x_{i_K(x)+1}(K)-)d_N^{-1}\left|\sum_{j\leq n}H_m(X_j)\right|\notag\\
&\leq\left|w(x_{i_K(x)+1}(K))S_N(n,x_{i_K(x)}(K),x_{i_K(x)+1}(K)-)\right|\notag\\
&+2nd_N^{-1}2^{-K}+2d_N^{-1}2^{-K}\left|\sum_{j\leq n}H_m(X_j)\right|.\label{abschaetzung3}
\end{align} 
Because of \eqref{abschaetzung2}, \eqref{abschaetzung3} and $\sum_{k=0}^{\infty}\varepsilon/(k+3)^2\leq \varepsilon/2$ the probability 
$P(\sup|w(x)S_N(n,x)|>\varepsilon)$ is dominated by
\begin{align}
 &P\left(\max_{x>0}|w(x_{i_0(x)+1}(0))S_N(n,x_{i_0(x)}(0))|>\varepsilon/9\right)\notag\\
+&\sum_{k=1}^{K}P\left(\max_{x>0}|w(x_{i_k(x)+1}(k))S_N(n,x_{i_{k-1}(x)}(k-1),x_{i_k(x)}(k))|>\varepsilon/(k+3)^2\right)\notag\\
+&P\left(\max_{x>0}|w(x_{i_K(x)+1}(K))S_N(n,x_{i_{K}(x)}(K),x_{i_K(x)+1}(K)-)|>\varepsilon/(K+3)^2\right)\notag\\
+&P\left(2d_N^{-1}2^{-K}\left|\sum_{j\leq n}H_m(X_j)\right|>\varepsilon/2-2nd_N^{-1}2^{-K}\right).\label{abschaetzung4}
\end{align}
Using \eqref{beschraenkt1} and Lemma \ref{momente} we get
\begin{align}
  {}&P\left(\max_{x\in\RR}\left|w(x_{i_0(x)+1}(0))S_N(n,x_{i_0(x)}(0))\right|>\frac{\varepsilon}{9}\right)\notag\\
\leq&\sum_{j=0}^{\infty}P\left(\left|w(x_{j+1}(0))S_N(n,x_j(0))\right|>\frac{\varepsilon}{9}\right)\notag\\
\leq&C\left(\frac{n}{N}\right)N^{-\gamma}81\varepsilon^{-2}\sum_{j=0}^{\infty}w(x_{j+1}(0))^2(1-F(x_{j}(0)))\notag\\
\leq&C\left(\frac{n}{N}\right)N^{-\gamma}81\varepsilon^{-2}.\label{abschaetzung5}
\end{align}
For $1\leq k<K$ we get by \eqref{beschraenkt2}
\begin{align}
 {}&P\left(\max_{x>0}\left|w(x_{i_{k+1}(x)+1}(k+1))S_N(n,x_{i_k(x)}(k),x_{i_{k+1}(x)}(k+1))\right|>\frac{\varepsilon}{(k+3)^2}\right)\notag\\
\leq&\sum_{j=0}^{\infty} P\left(|w(x_{j+2}(k+1))S_N(n,x_j(k+1),x_{j+1}(k+1))|>\frac{\varepsilon}{(k+3)^2}\right)\notag\\
\leq&C\left(\frac{n}{N}\right)N^{-\gamma}(k+3)^4\varepsilon^{-2}\sum_{j=0}^{\infty}w(x_{j+2}(k+1))^2(F(x_{j+2}(k+1))-F(x_{j}(k+1)))\notag\\
\leq&C\left(\frac{n}{N}\right)N^{-\gamma}(k+3)^4\varepsilon^{-2}\label{abschaetzung6}
\end{align}
and similarly
\begin{align}
{}&P\left(\max_{x>0}\left|w(x_{i_K(x)+1}(K))S_N(n,x_{i_K(x)}(K),x_{i_K(x)+1}(K)-)\right|>\frac{\varepsilon}{(K+3)^2}\right)\notag\\
\leq&C\left(\frac{n}{N}\right)N^{-\gamma}(K+3)^4\varepsilon^{-2}.\label{abschaetzung7}
\end{align}
We choose
\begin{equation*}
 K=\left\lfloor\log_2\left(\frac{8Nd_N^{-1}}{\varepsilon}\right)\right\rfloor+1,
\end{equation*}
which implies $\varepsilon/2-2Nd_N^{-1}2^{-K}\geq\varepsilon/4$ and therefore
\begin{align}
 &P\left(2d_N^{-1}2^{-K}\left|\sum_{j\leq n}H_m(X_j)\right|>\frac{\varepsilon}{2}-2nd_N^{-1}2^{-K}\right)\notag\\
\leq&P\left(d_N^{-1}\left|\sum_{j\leq n}H_m(X_j)\right|>\frac{\varepsilon}{4}2^{K-1}\right)\notag\\
\leq&\left(\frac{d_n}{d_N}\right)^2\left(\frac{\varepsilon}{4}\right)^{-2}2^{-2K+2}\notag\\
\leq&\left(\frac{d_n}{d_N}\right)^2d_N^2N^{-2}\notag\\
\leq&C\left(\frac{n}{N}\right)^{2-mD}\left(\frac{L(n)}{L(N)}\right)^mN^{-mD+\lambda}\notag\\
\leq&C\left(\frac{n}{N}\right)^{2-mD}N^{-mD+\lambda}\label{abschaetzung8}
\end{align}
for any $\lambda>0$. Remember that $P(\sup|w(x)S_N(n,x)|>\varepsilon)$ is dominated by \eqref{abschaetzung4}. Using \eqref{abschaetzung5},
\eqref{abschaetzung6}, \eqref{abschaetzung7} and \eqref{abschaetzung8}, this yields
\begin{align*}
 P\left(\sup_{x>0}\left|w(x)S_N(n,x)\right|>\varepsilon\right)\leq&C\left(\frac{n}{N}\right)N^{-\gamma}\varepsilon^{-2}\sum_{k=0}^K(k+3)^4+C\left(\frac{n}{N}
\right)^{2-mD}N^{-mD+\lambda}\\
\leq&C\left(\frac{n}{N}\right)N^{-\gamma}\varepsilon^{-2}(K+3)^5+C\left(\frac{n}{N}\right)^{2-mD}N^{-mD+\lambda}\\
\leq&CN^{-\rho}\left(\frac{n}{N}\varepsilon^{-3}+\left(\frac{n}{N}\right)^{2-mD}\right)
\end{align*}
for any $\rho$ with $0<\rho<\min(\gamma,mD-\lambda)$, because of
\begin{align*}
 (K+3)^5=&\left(\left\lfloor\log_2\left(8Nd_N^{-1}\varepsilon^{-1}\right)\right\rfloor+4\right)^5\\
\leq&C\left(\log(\varepsilon^{-1})+\log(CN)\right)^5\\
\leq&C\varepsilon^{-1}N^{\delta}
\end{align*}
for any $\delta>0$.\\
To prove the second case, i.e. $x\in(-\infty,0]$, we set
\begin{equation*}
 y_i(k):=\sup\{y\leq0:w(y)(\Lambda(0)-\Lambda(y))\geq i2^{-k}\}.
\end{equation*}
So we get corresponding versions of \eqref{abstand}, \eqref{beschraenkt1} and \eqref{beschraenkt2}, namely
\begin{align}
 &w(y_{j}(k))\Lambda(y_{j}(k),y_{j-1}(k)-)\notag\\
= &w(y_{j}(k))(-\Lambda(0)+\Lambda(y_{j-1}(k)-)+\Lambda(0)-\Lambda(y_{j}(k)))\notag\\
\leq &w(y_{j}(k))(\Lambda(0)-\Lambda(y_{j}(k)))-w(y_{j-1}(k)-)(\Lambda(0)-\Lambda(y_{j-1}(k)-))\notag\\
\leq &2^{-k},\\
{}\notag\\
 &\sum_{j=0}^{\infty}w(y_j(0))^2F(y_j(0))\notag\\
=&\sum_{j=0}^{\infty}\sum_{i=j}^{\infty}w(y_j(0))^2(F(y_i(0))-F(y_{i+1}(0)))\notag\\
=&\sum_{i=0}^{\infty}\sum_{j=0}^{i}w(y_j(0))^2(F(y_i(0))-F(y_{i+1}(0)))\notag\\
\leq&\Lambda(0)\sum_{i=0}^{\infty}w(y_{i+1}(0))^3(F(y_i(0))-F(y_{i+1}(0)))\notag\\
\leq &C\sum_{i=0}^{\infty}w(y_{i}(0))^3(F(y_i(0))-F(y_{i+1}(0)))\notag\\
<&\infty,\label{beschraenkt3}\\
{}\notag\\
&\sum_{i=0}^{\infty}w(y_{i+1}(k))^2(F(y_i(k))-F(y_{i+1}(k)))\notag\\
\leq&\sum_{i=0}^{\infty}w(y_{i+1}(0))^2(F(y_i(0))-F(y_{i+1}(0)))\notag\\
<&\infty.\label{beschraenkt4}
\end{align}
Now, for any $x\leq0$ and $K\in\NN$ we can find a chain
\begin{equation*}
 -\infty<y_{i_0(x)}(0)\leq y_{i_1(x)}(1)\leq\ldots\leq y_{i_K(x)}(K)\leq x,
\end{equation*}
with $y_{i_k(x)}(k)\leq x\leq y_{i_k(x)-1}(k)$. Using 
\begin{align*}
 &\left|w(x)S_N(n,x)\right|\\
\leq &|w(y_{i_0(x)}(0))S_N(n,y_{i_0(x)}(0))|+|w(y_{i_0(x)}(0))S_N(n,y_{i_0(x)}(0),y_{i_1(x)}(1))|\\
     &+|w(y_{i_1(x)}(1))S_N(n,y_{i_1(x)}(1),y_{i_2(x)}(2))|+\ldots+|w(x)S_N(n,y_{i_K(x)}(K),x)|
\end{align*}
and
\begin{align*}
&\left|w(x)S_N(n,y_{i_K(x)}(K),x)\right|\\
\leq &\left|w(y_{i_K(x)}(K))S_N(n,y_{i_K(x)}(K),y_{i_K(x)-1}(K)-)\right|+2nd_N^{-1}2^{-K}+2^{-K}d_N^{-1}\left|\sum_{j\leq n}H_m(X_j)\right|
\end{align*}
together with \eqref{beschraenkt3} and \eqref{beschraenkt4}, we can finish the proof in the same way as in the first case.
\end{proof}
We are now ready to prove the weighted weak reduction principle. Therefore we can use the original proof by Dehling and Taqqu. 
\begin{proof}[Proof of Theorem \ref{weak-reduction}]
Let $N=2^r$ and $M_N(n):=\sup_{x\in\RR}|w(x)S_N(n,x)|$. Using the stationarity of $(X_j)_{j\geq 1}$ we get for $n_1<n_2\leq N$
\begin{align*}
 M_N(n_1,n_2):=&M_N(n_2)-M_N(n_1)\\
\leq&\sup_{x\in\RR}|w(x)(S_N(n_2,x)-S_N(n_1,x))|\\
\stackrel{D}{=}&M_N(n_2-n_1)
\end{align*}
Together with Lemma \ref{lemma3} we obtain 
\begin{equation*}
 P\left(\max_{j=1,\ldots,2^{r-k}}\left|M_N((j-1)2^k,j2^k)\right|>\varepsilon\right)
\leq CN^{-\rho}(\varepsilon^{-3}+2^{(k-r)(1-mD)}).
\end{equation*}
Since $n=\sum_{k=0}^r\sigma_k2^{r-k}$, $\sigma_k\in\{0,1\}$, we have
\begin{equation*}
 M_N(n)=\sum_{k=0}^r\sigma_kM_N((j_k-1)2^{r-k},j_k2^{r-k}),
\end{equation*}
with some suitable $j_k\in\{1\ldots,2^k\}$. This yields  
\begin{align*}
 P\left(\max_{n\leq N}|M_N(n)|>\varepsilon\right)\leq&P\left(\sum_{k=0}^r\max_{j=1,\ldots,2^{r-k}}\left|M_N((j-1)2^k,j2^k)\right|>\varepsilon\right)\\
\leq&\sum_{k=0}^rP\left(\max_{j=1,\ldots,2^{r-k}}\left|M_N((j-1)2^k,j2^k)\right|>\varepsilon(k+2)^{-2}\right)\\
\leq&CN^{-\rho}\left(\varepsilon^{-3}\sum_{k=0}^{\log_2(N)}(k+2)^6+\sum_{k=0}^{\log_2(N)}2^{(k-r)(1-mD)}\right)\\
\leq&CN^{-\rho}(\varepsilon^{-3}+1).
\end{align*}
For $N\neq2^r$ we have $d_N^{-1}<Cd_{2^r}^{-1}$, with $r=\min\{r:N\leq 2^r\}$ and $C$ independent of $N$ and $r$. Hence
\begin{align*}
 P\left(\max_{n\leq N}|M_N(n)|>\varepsilon\right)\leq&P\left(\max_{n\leq 2^r}|M_{2^r}(n)|>C^{-1}\varepsilon\right)\\
\leq&C2^{-r\rho}(\varepsilon^{-3}+1)\\
\leq&CN^{-\rho}(\varepsilon^{-3}+1).
\end{align*}
\end{proof}
Before we can prove Theorem \ref{nclt} we need one last lemma. More precisely we have to show, that the function $J_m$, which yields the $m$th
order Hermite coefficient of $1_{\{Y_j\leq x\}}$, is bounded with respect to the weighted norm we use.
\begin{Lemma}\label{beschraenktlemma}
If $F$ has a finite $\delta$-th moment then for all $q\in\NN$ we have
\begin{equation*}
 \sup\limits_{x\in\RR}|w(x)J_q(x)|<\infty,
\end{equation*}
where $w(x)=(1+|x|)^{\lambda}$ and $\lambda=\delta/3$.
\end{Lemma}
\begin{proof}
Since $2\lambda<\delta$ and condition \eqref{delta} we have $E|w(Y_j)|^2<\infty$. The weight function $w$ is non-increasing on $(-\infty,0]$ and increasing on
$[0,\infty)$. Therefore we get for $x\leq0$ 
\begin{align*}
 |w(x)J_q(x)|&\leq E|w(x)1_{\{Y_j\leq x\}}H_q(X_j)|\\
             &\leq \sqrt{E|w(x)1_{\{Y_j\leq x\}}|^2E|H_q(X_j)|^2}\\
             &\leq \sqrt{E|w(Y_j)|^2E|H_q(X_j)|^2}\\
             &<\infty
\end{align*}
and for $x\geq 0$
\begin{align*}
 |w(x)J_q(x)|&=|E(w(x)(1-1_{\{Y_j\leq x\}})H_q(X_j))|\\
             &\leq \sqrt{E|w(x)1_{\{Y_j> x\}}|^2E|H_q(X_j)|^2}\\
             &\leq \sqrt{E|w(Y_j)|^2E|H_q(X_j)|^2}\\
             &<\infty.
\end{align*}
\end{proof}
\begin{proof}[Proof of Theorem \ref{nclt}]
 Since $C[0,1]$ is separable and $Z_m\in C[0,1]$ a.s., we can use the process convergence
\begin{equation*}
 Z_{m,N}(t):=d_N^{-1}\sum_{j\leq\lfloor Nt\rfloor}H_m(X_j)\stackrel{D}{\longrightarrow} Z_m(t)
\end{equation*}
in $D[0,1]$, investigated by \citet{Taqqu1975,Taqqu1979}, to apply the a.s. representation theorem \citep[page 71]{Pollard}. Therefore it exists processes
$(\tilde{Z}_{m,N}(t))_{t\in[0,1]}$ and $(\tilde{Z}_{m}(t))_{t\in[0,1]}$, with $(\tilde{Z}_{m,N}(t))\stackrel{D}{=}(Z_{m,N}(t))$,
$(\tilde{Z}_{m}(t))\stackrel{D}{=}(Z_{m}(t))$ and
\begin{equation*}
 \left\|\tilde{Z}_{m,N}(\cdot)-\tilde{Z}_{m}(\cdot)\right\|_{\infty}\longrightarrow0\textnormal{~~~a.s.}
\end{equation*}
Using Lemma \ref{beschraenktlemma} we have
\begin{equation*}
 \left\|J_m(\cdot)\tilde{Z}_{m,N}(\cdot)-J_m(\cdot)\tilde{Z}_{m}(\cdot)\right\|_{w}\longrightarrow0\textnormal{~~~a.s.}
\end{equation*}
and this implies
\begin{equation*}
 J_m(x)Z_{m,N}(t)\stackrel{D}{\longrightarrow} J_m(x)Z_m(t)
\end{equation*}
in $D_w\subset D([-\infty,\infty]\times[0,1])$, equipped with the weighted norm $\|\cdot\|_w$. Theorem \ref{nclt} follows by the weighted weak reduction
principle (Theorem \ref{weak-reduction}).
\end{proof}
\textbf{Acknowledgements.}~~ I would like to thank both referees for carefully reading my manuscript and  for their helpful comments.
\bibliography{literatur}
\end{document}